\documentclass[12pt]{amsart}

\usepackage[T1]{fontenc}

\usepackage{tikz}
\usetikzlibrary{calc, through}
\usepackage{xy, graphicx, color, hyperref}
\usepackage{float}
\usepackage{tikz-cd} 
\usepackage{mathrsfs}
\usepackage{amsmath}
\usepackage{amssymb}
\usepackage{amsfonts, ytableau}

\newcommand{\multi}[2]{\bigg(\!\!\dbinom{#1}{#2}\!\!\bigg)}

\usepackage{tikz,verbatim}

\makeatletter
\def\subsection{\@startsection{subsection}{3}%
\z@{.9\linespacing\@plus.7\linespacing}{.1\linespacing}%
{\normalfont\bfseries}}
\makeatother

\title[CRT for Partitions]{A Chinese Remainder Theorem for Partitions} 

\author{Seethalakshmi K and Steven Spallone \\
 \bigskip
\today}

\theoremstyle{definition}
\newtheorem{theorem}{Theorem}
\newtheorem{corollary}{Corollary}
\newtheorem{lemma}{Lemma}

\newtheorem{prop}{Proposition}
\newtheorem{defn}{Definition}

\newcommand{\nc}{\newcommand}
\newtheorem{ex}{{\bf Example}}

\nc{\thm}{\theorem}
\nc{\cor}{\corollary}
\nc{\mc}{\mathcal}
\nc{\mb}{\mathbb}
\nc{\mf}{\mathfrak}
\nc{\ul}{\underline}
\nc{\ol}{\overline}
\nc{\N}{\mb N}
\nc{\R}{\mb R}
\nc{\Z}{\mb Z}
\nc{\Q}{\mb Q}
\nc{\C}{\mb C}

\usetikzlibrary{quotes,angles,positioning}

\newcommand{\Zb}{\mb Z / b \mb Z}

\newcommand{\Zd}{\mb Z / d \mb Z}
\newcommand{\Zt}{\mb Z / t \mb Z}
\newcommand{\Zs}{\mb Z / s \mb Z}
\newcommand{\Zm}{\mb Z / m \mb Z}
\newcommand{\SdTdk}{\inlinebnm{S_j \times T_j}{k_j}}
\newcommand{\Sdk}{\inlinebnm{S_j}{k_j}}
\newcommand{\Tdk}{\inlinebnm{T_j}{k_j}}
\newcommand{\SkDkTk}{\inlinebnm{S}{k} \times_{\inlinebnm{D}{k}} \inlinebnm{T}{k}}
\newcommand{\SDTk}{\inlinebnm{S \times_D T}{k}}

\newcommand{\STk}{\inlinebnm{S \times T}{k}}
\newcommand{\Mk}{\inlinebnm{M}{k}}
\newcommand{\Sk}{\inlinebnm{S}{k}}
\newcommand{\Tk}{\inlinebnm{T}{k}}
\newcommand{\Dk}{\inlinebnm{D}{k}}
\newcommand{\ms}{\mathscr}

\newcommand{\A}{\mathscr{A}}

\newcommand{\Nsk}{N_{\sigma}(k)}

\newcommand{\Nstk}{N_{\sigma, \tau}(k)}

\newcommand{\Zak}{\inlinebnm{\mb Z/a \mb Z}{k}}
\newcommand{\Zbk}{\inlinebnm{\mb Z/b \mb Z}{k}}
\newcommand{\Zmk}{\inlinebnm{\mb Z/m \mb Z}{k}}

\newcommand{\Ztk}{\inlinebnm{\mb Z/t \mb Z}{k}}
\newcommand{\Zsk}{\inlinebnm{\mb Z/s \mb Z}{k}}

\newcommand{\conv}{\mathrm{conv}}

\nc{\W}{\ul \N}

\newcommand{\inlinebnm}[2]{\ensuremath{\big(\!\tbinom{#1}{#2}\!\big)}}

\theoremstyle{definition}

\newtheorem{remark}{Remark}

\nc{\dmo}{\DeclareMathOperator}

\nc{\mat}[4]{
\begin{pmatrix}
#1 & #2 \\
#3 & #4
\end{pmatrix}
}
\dmo{\An}{A}
\dmo{\ASpan}{ASpan}
\dmo{\Span}{Span}
\dmo{\id}{id}
\dmo{\Char}{char}
 \dmo{\cone}{cone}
\dmo{\Ker}{Ker} 
\dmo{\val}{val} 
\dmo{\ord}{ord}
\dmo{\pr}{pr}
\dmo{\I}{I}
\dmo{\II}{II}
\dmo{\Match}{Match}
\dmo{\odd}{odd}
\dmo{\sgn}{sgn}
\nc{\beq}{\begin{equation*}}
\nc{\eeq}{\end{equation*}}
\nc{\half}{\frac{1}{2}}
\dmo{\Mod}{mod}
\dmo{\core}{core}
\dmo{\Core}{Core}
\dmo{\res}{res}
\dmo{\lin}{lin}
\dmo{\lcm}{lcm}
\dmo{\vol}{vol}
\dmo{\fin}{fin}
\dmo{\M}{Mat}
 
\nc{\la}{\lambda}

\nc{\lip}{\langle}
\nc{\rip}{\rangle}
\dmo{\supp}{supp}
\dmo{\Ima}{Im}
\dmo{\Res}{Res}
\dmo{\Ind}{Ind}
\dmo{\tr}{tr}
\dmo{\Sym}{Sym}
\dmo{\reg}{reg}
\dmo{\End}{End}
\dmo{\Hom}{Hom}
\dmo{\im}{im}

\address{Mathematical Sciences Institute, The Australian National University, Canberra, Australia}    
\email{seethalakshmi.kayanattath@anu.edu.au}
\address{Indian Institute of Science Education and Research, Pune-411021,Maharashtra,India}
\email{sspallone@gmail.com}
\date{\today}
\keywords{t-core partitions, Ehrhart's Theorem, transportation polytopes}
 
\begin{document}
\maketitle
\begin{center} \date{}
\end{center}
\begin{abstract}
Let $s,t$ be natural numbers, and fix an $s$-core partition $\sigma$ and a $t$-core partition $\tau$. Put $d=\gcd(s,t)$ and $m=\lcm(s,t)$, and write $\Nstk$ for the number of $m$-core partitions of length no greater than $k$ whose $s$-core is $\sigma$ and $t$-core is $\tau$. We prove that for $k$ large, $\Nstk$ is a quasipolynomial of period $m$ and degree $\frac{1}{d}(s-d)(t-d)$.
\end{abstract}
\tableofcontents

\section{Introduction}

For a partition $\lambda$ and a natural number $t$, the $t$-core of $\la$, written $\core_t \la$, is obtained by removing hooks of size $t$ from the Young diagram of $\la$, until no more can be removed. This analogue of the Division Algorithm has its origins in the representation theory of the symmetric group \cite{Nakayama}, and finds application in the study of the partition function \cite{wildon2008counting}. We present an analogue of the Chinese Remainder Theorem in this paper.

Write $\mc C_t$ for the set of $t$-cores.  Suppose $s, t \in \mb N$ are relatively prime, and consider the map 
\beq
\core_{s, t}: \mc C_{st} \to \mc C_s \times \mc C_t
\eeq
taking $\la$ to $(\core_s \la,\core_t \la)$. This map $\core_{s,t}$ is surjective (\cite{fayers2014generalisation}, Section 5.1), but far from injective. In fact the fibres are infinite. To capture their behavior we stratify them by length as follows.

 Given a partition $\la$, write $\ell(\lambda)$ for its length, meaning the number of parts of $\la$.   For a fixed $(\sigma, \tau) \in \mc C_s \times \mc C_t$, let $m =  st$ and put
 \begin{equation} \label{defn.of.N}
\Nstk= \#\{ \lambda \in \mc C_{m} \mid \core_s \la=\sigma, \: \core_t \la= \tau, \: \ell(\lambda) \leq k \}.
\end{equation}
In other words, $\Nstk$ is the cardinality of the $k$th stratum,
\begin{equation}
\label{eq: fibre of core s,t}
\core_{s,t}^{-1}(\sigma, \tau) \cap \mc C_m^k
\end{equation}
where $\mc C_m^k$ is the set of $m$-cores of length no greater than $k$.

Our first result is:
\begin{thm} \label{main.thm.intro} Let $s, t \in \mb N$ be relatively prime.  There is a quasipolynomial $Q_{\sigma,\tau}(k)$ of degree $(s-1)(t-1)$ and period $st$, so that for integers
 $k \gg 0$, we have $\Nstk=Q_{\sigma,\tau}(k)$.  The leading coefficient of $Q_{\sigma,\tau}(k)$ is a positive number $V_{s,t}$ depending only on $s$ and $t$.
 \end{thm}

\bigskip

Here a ``quasipolynomial of period $n$'' is a function on natural numbers whose restriction to each coset $n \N +i$ is a polynomial; see Section ~\ref{sec:Counting Fibres of Core Maps}. The quantity $V_{s,t}$ is the volume of a certain polytope we define in Section ~\ref{sec: proofs of thm 1 and 3}.

\begin{remark}
If $\sigma = \tau$, then $\sigma$ is simultaneously an $s$-core and a $t$-core. In fact, the intersection $\mc C_s \cap \mc C_t$ is finite and well-studied; see  \cite{anderson2002partitions} and \cite{fayers2011t}.
\end{remark}

 Our method is as follows. We associate to $\tau \in \mc C_t^k$ a multiset $H_{\tau,t}^k$ on $\{0,1,\ldots, t-1\}$ of size $k$, corresponding to the first column hook lengths of $\tau$ modulo $t$.  (This is James' theory of abacuses \cite[page 78]{james1984representation}.) Members of \eqref{eq: fibre of core s,t} correspond to matchings between $H_{\sigma,s}^k$ and $H_{\tau,t}^k$. (See Section \ref{sec: Counting fibres of the Sun Tzu map}.)

Generally, let  $F,G$ be multisets of the same size, with multiplicity vectors $\vec F$, $\vec G$.   Write $\mc M(\vec F,\vec G)$ for the polytope of real matrices with nonnegative entries, row margins $\vec F$, and column margins $\vec G$.  Then the matchings between $F$ and $G$ correspond bijectively with the integer points of $\mc M(\vec F,\vec G)$. 

In our situation, the polytopes $\mc M$ grow linearly in $k$, and we may apply Ehrhart's theory, which says that if $\mc P$ is a polytope with integer vertices, then the number of integer points in $n \cdot \mc P$ is a polynomial in ~$n$. We refer the reader to Section 4.6.2 of \cite{Stanley}, and Chapter 3 of \cite{beck2007computing}. The degree of the polynomial is the dimension of $\mc P$, and its leading coefficient is the relative volume of $\mc P$.

 When $\sigma=\tau=\emptyset$, and $k$ is a multiple of $st$, this directly gives our result.  The technical heart of this paper is extending Ehrhart's Theorem to all fibres and all $k$.

The polytopes $\mc M(\vec F,\vec G)$ arising from row/column constraints are called ``transportation polytopes''.  Each can be expressed in the form
\beq
\mc P(A,\vec b)=\{ \vec x \mid A \vec x \leq \vec b, \vec x \geq 0 \},
\eeq
for some \emph{totally unimodular matrix} $A$, meaning that all the minors of ~$A$ are either $0,1$, or $-1$.  Write $N(A, \vec b)$ for the number of integer points in the polytope $\mc P(A, \vec b)$. Our extension of Ehrhart's Theorem is: 
  
\begin{theorem}  \label{ext.thm.intro} Let $A$ be an $m \times n$ totally unimodular matrix and $\vec b, \vec c \in \mb Z^m$. Suppose  $A$ does not have any zero rows, and that $\mc P(A, \vec b)$ is bounded of dimension $n$. Then there is a polynomial $f(k)$ so that for integers $k \gg 0$, we have $N(A, \vec bk+\vec c)=f(k)$. Moreover $\deg f=n$ and the leading coefficient of $f$ is the volume of $\mc P(A,\vec b)$.
\end{theorem}
Note that Ehrhart's Theorem gives the case $\vec c = 0$.
\bigskip

Next, suppose $s$ and $t$ are not relatively prime.  Let $d=\gcd(s,t)$, $m=\lcm(s,t)$ and $\ell_0 = \max(\ell(\sigma), \ell(\tau))$. Again define $\Nstk$ by \eqref{defn.of.N}. 
\begin{theorem}   \label{not.prime.intro}
If $\core_d(\sigma) = \core_d(\tau)$, then there is a quasipolynomial $Q_{\sigma, \tau}(k)$ of degree $\frac{1}{d} (s-d)(t-d)$ and period $m$, so that for integers $k \gg 0$, we have $\Nstk=Q_{\sigma,\tau}(k)$. The leading coefficient of $Q_{\sigma, \tau}(k)$ is $\left(V_{\frac{s}{d},\frac{t}{d}}\right)^d$.
\end{theorem}

(It is easy to see that if $\core_d(\sigma) \neq \core_d(\tau)$, then each $\Nstk=0$.)

We mention also a simpler related result for the fibres of the map $\core_s : \mc C_{st} \rightarrow \mc C_s$ taking an $st$-core  to its $s$-core. For $\sigma \in \mc C_s$,
 let $$N_{\sigma}(k)=\{\lambda \in \mc C_{st} \mid \core_s \lambda = \sigma, \ell(\lambda) \leq k\}.$$  

\begin{thm} \label{simpler.related}
There is a quasipolynomial $Q_{\sigma}(k)$ of degree $s(t-1)$ and period $s$, so that for $k \geq \ell(\sigma)$, we have $\Nsk=Q_{\sigma}(k)$. The leading coefficient of $Q_\sigma(k)$ is $\dfrac{1}{(t-1)!^s}$.
\end{thm}

The layout of this paper is as follows. In Section \ref{sec: preliminaries}, we recall terminology for partitions and multisets, and in Section \ref{sec: cores}, we review James' Theory of abacuses for computing $t$-cores.
Theorem \ref{simpler.related} is worked out in Section \ref{sec:Counting Fibres of Core Maps}, as a warm up to later material. 
 Section \ref{sec: Converting to a Multiset Matching Problem}  converts the fibre counting problem into a multiset matching problem. 
Preliminaries for polytopes are given in Section  \ref{sec: Preliminaries of polytopes and notation}. 
Our theory of integer points in polytopes, including Theorem \ref{ext.thm.intro}, is contained in Section  \ref{sec: Counting integer points in polytopes}. 
Finally in Section \ref{sec: proofs of thm 1 and 3} we apply this to our core problem, giving Theorems \ref{main.thm.intro} and \ref{not.prime.intro}.

\section*{Acknowledgments}
The authors would like to thank Amritanshu Prasad for useful conversations on totally unimodular matrices and polytopes, Brendan McKay for pointing us to Ehrhart theory, and Jyotirmoy Ganguly, Ojaswi Chaurasia and Ragini Balachander for helpful discussions.

\section{Preliminaries} \label{sec: preliminaries}

Write $\N=\{1,2,\ldots \}$ for the set of \emph{natural} numbers and $\W=\{0,1,2, \ldots \}$ for the set of \emph{whole} numbers. If $S$ is a set, write `$\id_S$' for the identity map.
If $f: S \to T$ is a map, write `$\im f$' for the image of $f$. We write $\wp_{\fin}(S)$ for the set of finite subsets of $S$.

\subsection{Multisets}

Let $S$ be a set. When $S$ is a finite set, we write either `$\#S$' or `$|S|$' for the cardinality of $S$. For $k \in \W$, write $\binom{S}{k}$ for the set of $k$-element subsets of $S$.  Note that $\# \binom{S}{k}=\binom{\# S}{k}$.

A \emph{multiset on $S$} is a function $F$ from $S$ to $\ul{\mb N}$. The \emph{cardinality} of $F$ is the sum
\beq
|F|=\sum_{s \in S} F(s).
\eeq
The \emph{support} of a multiset $F$ is
\beq
\supp(F)=\{s \in S \mid F(s) \neq 0 \}.
\eeq

Write `$\inlinebnm{S}{k}$' for the set of multisets on $S$ of cardinality $k$. Note that $\# \inlinebnm{S}{k}= \inlinebnm{\# S}{k}$, where
 \beq
 \multi{n}{k}= \binom{k+n-1}{k}.
 \eeq

Write $\mc M_{\fin}(S)$ for the set of multisets on $S$ with finite support. Thus
\beq
\mc M_{\fin}(S)= \bigcup_{k \geq 0} \multi{S}{k}.
\eeq

Given finite sets $S, T$ and a map $f: S \to T$, define $f_* : \mc M_{\fin}(S) \to \mc M_{\fin}(T)$ by 
\beq
f_*(F)(t) = \sum\limits_{s \in f^{-1}(t)}F(s).
\eeq
We use the same notation to denote the restriction  $f_* : \Sk \to \Tk$, when $k$ is understood.

 \begin{lemma}\label{fibres.for.multisets} 
For $G \in \mc M_{\fin}(T)$, we have
\beq
\# (f_*)^{-1}(G)=\prod_{t \in T}  \inlinebnm{\#f^{-1}(t)}{G(t)}.
\eeq
\end{lemma}

\begin{proof}
We need to count the $F \in \mc M_{\fin}(S)$ such that for all $t \in T$, we have
\beq
\begin{split}
G(t) &= f_*(F)(t) \\
	&=  \sum_{s \in f^{-1}(t)} F(s). \\
	\end{split}
	\eeq
 
There are $ \inlinebnm{\#f^{-1}(t)}{G(t)}$ many choices for the values of $F$ on the fibre over $t \in T$. Multiplying these gives the formula.
\end{proof}

Note that
\beq
\im f_*= \{H \in  \mc M_{\fin}(T) \mid \supp(H) \subseteq \im(f)\}.
\eeq

For maps $f: S \to T$ and $g: T \to U$, note that $\left(g \circ f\right)_* = g_* \circ f_*,$ and $\left(\id_S\right)_* = \id_{\Sk}$. This makes the association $S \leadsto  \mc M_{\fin}(S)$ a functor from the category of sets to itself.

\subsection{Partitions and Pseudopartitions}

A \emph{partition} $\la$ is a weakly decreasing finite sequence of natural numbers. Thus $\lambda = \left( a_1, a_2, \ldots, a_\ell \right)$, with $a_1 \geq a_2 \geq \cdots \geq a_\ell>0$. 
In these terms $a_1+a_2+ \cdots +a_\ell$ is the \emph{size} of $\la$, and $\ell=\ell(\la)$ is the \emph{length} of ~$\la$. 
We allow the empty partition $\la=\emptyset$; its length is $0$.
Write $\Lambda$ for the set of partitions, and $\Lambda^\ell$ for the set of partitions of length $\ell$.

We define \emph{pseudopartitions} to be weakly decreasing finite sequences of \emph{whole} numbers. Thus $\lambda = \left( a_1, a_2, \ldots, a_\ell \right)$, with $a_1 \geq a_2 \geq \cdots \geq a_\ell \geq 0$.  Again, $\ell$ is the \emph{length} of $\la$. For instance $\la=(5,4,3,1,0,0)$ is a pseudopartition of length $6$, with $2$ ``trailing zeros''. Write $\ul \Lambda$ for the set of pseudopartitions, and $\ul \Lambda^k$ for the set of pseudopartitions of length $k$. 
Let $z: \ul \Lambda \to \ul \Lambda$ be the map which adds a trailing zero to the end of a pseudopartition. For instance $z((5,4,0))=(5,4,0,0)$. 
 If $\la$ is a pseudopartition of length $\ell \leq k$, define 
\beq
u^k(\la)=z^{k-\ell}(\la) \in  \ul \Lambda^k.
\eeq
Write $r: \ul \Lambda \to \Lambda$ for the map which removes all trailing zeros from the pseudopartition to make it a partition, e.g., $r((5,4,3,1,0,0))=(5,4,3,1)$.
The fibres of $r$ are the $z$-orbits of partitions.

\subsection{Young Diagrams}

The \emph{Young diagram} of a partition $\lambda = \left( a_1, a_2, \ldots, a_\ell \right)$ is given by
$$\mc Y(\lambda) = \{ (i,j) \in \mb N \times \mb N \mid 1 \leq i \leq \ell, 1 \leq j \leq a_i \}.$$
It is visualized as a collection of left justified cells arranged in rows with $a_i$ cells in the $i$-th row. 
\begin{ex} \label{YT.ex}
Here is $\mc Y((5,4,3,1))$:
\begin{center}
\ytableausetup{centertableaux}
\ytableaushort
{,\none , \none}
* {5,4,3,1}
\end{center}
\end{ex}

The \emph{hook} $\mf h_c$ associated to a cell $c$ of $\mc Y(\lambda)$ consists of all cells to the right of $c$ and below $c$, together with $c$ itself. The \emph{hooklength} is the total number of cells in the hook. In the above diagram, the hooklength for the $(1,1)$-cell is $8$.
A hook with hooklength $t$ is called a \emph{$t$-hook}. 

\begin{remark} It would be more consistent to say ``hooksize'' rather than ``hooklength'', but the usage is standard.
\end{remark}

Of particular importance are the hooklengths corresponding to cells in the first column of $\mc Y(\la)$; we can use these to reconstruct $\la$. The set of first column hooklengths in Example \ref{YT.ex} is 
$\{ 8,6,4,1 \}$.  

\subsection{Beta sets}\label{subsec: Beta sets}

The map which takes a partition $\la$ to the set of first column hooklengths of  $\mc Y(\la)$ gives a bijection between $\Lambda$ and $\wp_{\fin}(\N)$.   In this paragraph  we extend this to a bijection between $\ul \Lambda$ and $\wp_{\fin}(\W)$. 

Define $\beta: \ul \Lambda \to \wp_{\fin}(\W)$ by
\beq
\beta((a_1, a_2, \ldots, a_\ell))= \{a_1 + (\ell-1), a_2 + (\ell-2), \ldots, a_\ell\}.
\eeq
 The inverse $\beta^{-1} :\wp_{\fin}(\W) \to  \ul \Lambda$ is given by
\beq
\beta^{-1}(\{ h_1, \ldots, h_\ell\})= (h_1-(\ell-1), h_2-(\ell-2), \ldots, h_\ell),
\eeq
for $h_1>\cdots >h_\ell \geq 0$. When $\lambda$ is a partition, $\beta(\lambda)$ is the set of first column hooklengths of $\mc Y(\lambda)$. We also write `$H_\la$' for $\beta(\lambda)$. 
 
The  ``add a trailing zero'' map $z$ translates under $\beta$ to
 \beq
 Z=\beta \circ z \circ \beta^{-1} : \wp_{\fin}(\W) \to \wp_{\fin}(\W),
 \eeq
which comes out to be
\beq
\{x_1, \ldots, x_\ell \} \mapsto \{x_1+1, \ldots, x_\ell+1, 0\}.
\eeq
 Similarly we can translate $u^k$ to 
 \beq
 U^k=\beta \circ u^k \circ \beta^{-1}.
 \eeq

\begin{defn} Let $\la$ be a partition. A  \emph{beta set} of $\la$ is a set of the form $Z^j(\beta(\la))$ for some $j \in \W$.
\end{defn}

\begin{ex}
Let $\lambda = (5,4,3,1)$. Then $H_{\lambda} = \{8,6,4,1\}$, so the beta sets of $\lambda$ are:
$$\{8,6,4,1\} \overset{Z} {\mapsto}\{9,7,5,2,0\} \overset{Z}{\mapsto} \{10,8,6,3,1,0\}  \overset{Z}{\mapsto} \ldots$$ 
\end{ex}

 \begin{defn} If $\la$ is a partition of length $\ell \leq k$, write $H_\la^k$ for the beta set of $\la$ with cardinality $k$, i.e.,
\beq
H_\la^k=U^k(H_\la)=Z^{k-\ell}(H_\la).
\eeq
\end{defn}

In the example above, $H_\la^6=\{10,8,6,3,1,0\}$.

The $\beta$-set version of the ``remove all trailing zeros'' retraction $r$ is
\beq
R= \beta \circ r \circ \beta^{-1}:  \wp_{\fin}(\W) \to \wp_{\fin}(\N).
\eeq

It can be computed directly as follows. Given $X \in  \wp_{\fin}(\W)$ with $0 \in X$, put 
\beq
m=\min(\W-X),
\eeq
i.e., the smallest whole number not in $X$. Then $R(X)$ is obtained by removing $\{0, \ldots, m-1\}$ from $X$ and subtracting $m$ from the remaining members.
 Of course, if $0 \notin X$, then $R(X)=X$.

\section{Cores} \label{sec: cores}

\begin{defn}
A \emph{$t$-core} is a partition having no $t$-hook in its Young diagram. Write $\mc C_t$ for the set of $t$-cores, and put
\beq
\mc C_t^k=\{ \la \in \mc C_t \mid \ell(\la) \leq k \}.
\eeq
\end{defn}

\subsection{Hook Removal}

Suppose $\la$ is a partition whose Young diagram contains a $t$-hook $\mf h$.  One may remove $\mf h$ from $\mc Y(\la)$ by simply deleting $\mf h$, then moving any disconnected cells one unit up and one unit to the left. 

\begin{ex}\label{hookremoval.ex}
The removal of a 6-hook from the partition $(5,4,3,1)$:
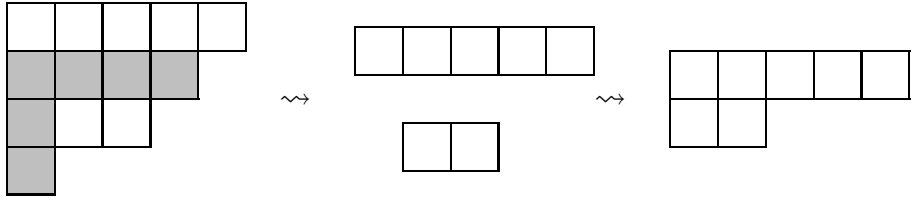
\begin{center}
\ydiagram[*(white)]
{5,0,1+2,0}
*[*(black!25)]{5,4,3,1}
	\hspace*{0.3cm}
 $\leadsto$
 \hspace*{0.3cm}
 \ydiagram{5,0,1+2}
 $\leadsto$
  \hspace*{0.3cm}
 \ydiagram{5,2}

\end{center}
\end{ex}

In fact, if we remove  $t$-hooks successively until no $t$-hook remains, the final Young diagram does not depend on the choices of hooks at each step.   The corresponding partition is called the \emph{$t$-core of $\lambda$}, and denoted `$\core_t \la$'.

\begin{ex} \label{core.ex}
We remove $6$-hooks successively from the Young diagram of $(5,4,3,1)$ to obtain its $6$-core, which is the partition $(1)$. In Example~\ref{hookremoval.ex} we removed a 6-hook from $(5,4,3,1)$ to get $(5,2)$, continuing from there we remove the remaining 6-hook:
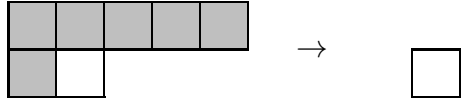
\vspace{0.1 cm}
\begin{center}
\ydiagram[*(white)]
{0,1+1}
*[*(black!25)]{5,2}
\hspace*{0.5cm}
 $\to$
 \hspace*{0.2cm}
 \ydiagram{0,1+1}
\end{center}
\end{ex}

\begin{lemma} \label{hook.removin.lemma} Let $\la$ be a partition, and $X=\{x_1, \ldots, x_k \}$  a $\beta$-set for $\la$.
Then $\mc Y(\la)$ has a $t$-hook iff $\exists \hspace{0.1cm} 1 \leq i \leq k$ so that $x_i \geq t$ and $x_i-t \notin X$. In this case, there is a $t$-hook $\mf h$ of $\mc Y(\la)$ so that
\begin{equation} \label{up}
\{ x_1, \ldots, x_i-t, \ldots, x_k \}
\end{equation}
is a $\beta$-set for $\la \backslash \mf h$.
\end{lemma}

\begin{proof} This is  \cite[Lemma 2.7.13]{james1984representation}.
\end{proof}

\begin{ex} The sequence of hook removals in Example \ref{core.ex} corresponds to the sequence 
\beq
\{8,6,4,1\} \leadsto R(\{8,0,4,1\})=\{6,2\} \leadsto R(\{0,2\})=\{1\}
\eeq
of $\beta$-sets.
\end{ex}

\subsection{$\beta$-set Version}

Fix $t \geq 1$.

 \begin{defn} A subset $X$ of $\W$ is \emph{$t$-reduced}, provided whenever $x \in X$ with $x \geq t$, we have $x-t \in X$. Write $\mc R_t$ for the set of finite $t$-reduced subsets of $\N$, and $\ul{\mc R}_t$ for the set of finite $t$-reduced subsets of $\W$.
 \end{defn}
For example, $X=\{0,1,3,4,6\}$ is $3$-reduced but not $2$-reduced.
We  view $\wp_{\fin}(\W)$, and thus $\ul {\mc R}_t$, inside  $\mc M_{\fin}(\W)$ via recognizing a subset of ~$\W$ as a multiset on $\W$.
Note that the retraction $R$ maps $\ul{\mc R}_t$ to $\mc R_t$.

 Let
 \beq
 f_t: \W \to \Z/t\Z
 \eeq
 be the usual  remainder mod $t$ map, and let
 \beq
 \rho_t=(f_t)_*: \mc M_{\fin}(\W) \to \mc M_{\fin}( \Z/t\Z)
 \eeq
be the induced map on multisets.

The subset $\ul{\mc R}_t$ is a transversal for $\rho_t$, in the sense that for each $F \in \mc M_{\fin}(\W)$ there is a unique $F_0 \in \ul{\mc R}_t$ with $\rho_t(F)=\rho_t(F_0)$.

\begin{remark}
For $X \in \wp_{\fin}(\W)$, the map $X \mapsto X_0$ is traditionally visualized in terms of a base $t$ abacus, having $t$ runners labeled $0$ to $t-1$, on which beads may be stacked. Given $X$, beads are arranged on the abacus corresponding to their base $t$ place value of members of $X$. To bring $X$ to $t$-reduced form, one simply slides the beads up. This is James' abacus method from \cite[page 78]{james1984representation}.
 
\bigskip

For example, given $X = \{8,6,4,1\}$ and $t=6$,  the abacus method  
\begin{center}
$\begin{array}{cccccc}
\ol{0} & \ol{1} & \ol{2} & \ol{3} & \ol{4} & \ol{5}\\
\hline
\circ & \bullet & \circ & \circ & \bullet & \circ\\
\bullet & \circ & \bullet & \circ & \circ & \circ
\end{array} \hspace{1cm} \leadsto \hspace{1cm}
\begin{array}{cccccc}
\ol{0} & \ol{1} & \ol{2} & \ol{3} & \ol{4} & \ol{5}\\
\hline
\bullet & \bullet & \bullet & \circ & \bullet & \circ\\
\circ & \circ & \circ & \circ & \circ & \circ\\
\end{array}$
\end{center}
gives $X_0 = \{4,2,1,0\}.$

\end{remark}

The map $\rho_t$ has a ``minimal'' section 
\beq
c_t:  \mc M_{\fin}(\Z/t\Z) \to  \mc M_{\fin}(\W)
\eeq
with image equal to  $\ul{\mc R}_t$, described as follows.
Given $F \in  \mc M_{\fin}(\Z/t\Z)$,  put
\beq
c_t(F)=\bigcup_{i=0}^{t-1} \{ at+i \mid 0 \leq a \leq F(i) \}.
\eeq
Thus the map $F \mapsto F_0$ above is $c_t \circ \rho_t$.

\begin{defn} Given $X \in \wp_{\fin}(\W)$, put
\beq
\Core_t(X)=R(c_t(\rho_t(X))).
\eeq
\end{defn}

\begin{prop} \label{core.beta} If $\la$ is a pseudopartition, then
\beq
\core_t(r(\la)) =\beta^{-1}( \Core_t(\beta(\la))). 
 \eeq
\end{prop}

\begin{proof} 
Let $X=\beta(\la)$, and suppose $X_0$ is obtained from $X$ by a sequence of steps, where each step replaces some $x_i \geq t$ with $x_i-t$, so long as $x_i-t \notin X$.
By Lemma \ref{hook.removin.lemma}, $X_0$ is a $\beta$-set for $\core_t(r(\la))$.

By construction, $X_0 \in \ul{\mc R}_t$, with $\rho_t(X)=\rho_t(X_0)$. By the above, we must have $c_t (\rho_t(X))=X_0$.
It follows that $R(c_t(\rho_t(X)))=\beta(\core_t(r(\la)))$, and the proposition follows.
\end{proof}

\begin{ex}
Let $t=6$. For $\lambda = (5,4,3,1)$, $X = H_\lambda = \{8,6,4,1\}$.
Put $F=\rho_t(X)$. Then $\supp F=\{0,1,2,4\}$, and $F$ takes value $1$ at each point in its support.
 Thus
\beq
c_t(X) = \{0\} \cup \{1\} \cup \{2\} \cup \{4\}=\{4,2,1,0\},
\eeq
and therefore
\beq
\begin{split}
 \core_6(5,4,3,1) &= \beta^{-1}(R(\Core_6(X))) \\
 			&= \beta^{-1}(\{ 1\}) \\
  &= (1).
  \end{split}
 \eeq
\end{ex}

For a partition $\lambda$ of length no greater than $k$, let
\begin{equation} \label{defn.of.H.thing}
H_{\la,t}^k=\rho_t(H_\la^k) \in  \Ztk.
\end{equation}

\begin{prop}\label{prop: cores multisets bijection}
The map $\A_t : \mc C_t^k \to \Ztk$ taking $\la \mapsto H_{\la, t}^k$ is a bijection. Thus
\beq
\#  \mc C_t^k=\binom{k+t-1}{k}.
\eeq
\end{prop}
(Compare {\cite[Theorem~2.4]{zhong2020bijections}.)

\begin{proof} 
We have
\beq
\A_t=\rho_t \circ \beta \circ u^k.
\eeq
Let us see that 
\begin{equation} \label{inv.of.at}
\A_t^{-1}=r \circ \beta^{-1} \circ c_t: \Ztk \to \mc C_t^k.
\end{equation}
is in fact inverse to  $\A_t$.  On the one hand, 
\beq
\begin{split}
\rho_t \: \beta \: u^k \: r \: \beta^{-1}\:  c_t &= \rho_t \: U^k\:  R\:  c_t \\
	&= \rho_t \: c_t \\
	&= \id \text{ on $\Ztk$}. \\
	\end{split}
	\eeq
 On the other hand,
 \beq
 \begin{split}
 r \:\beta^{-1} \: c_t \: \rho_t \: \beta \: u^k &= \beta^{-1} \: R \:c_t \:\rho_t\:  \beta \: u^k \\
 &= \core_t \circ \: r \: u^k \\
 &=\core_t \\
 &= \id \text{ on $\mc C_t^k$}.  \qedhere
 \end{split} 
 \eeq 
\end{proof}

Note that
\beq
H^k_{\core_t \la,t}=H_{\la,t}^k.
\eeq
 
\begin{lemma}
\label{lemma: periodicity of betaset mod s set}
For $i \geq \ell(\lambda)$, we have
$$H_{\lambda,t}^{i+t}(j) = H_{\lambda,t}^i(j) + 1.$$
\end{lemma}

By definition,
\beq
\begin{split}
H_{\lambda}^{i+t} &= Z^t(H_\la^i) \\
				&=  (H_{\lambda}^i +t) \cup \{t-1, t-2, \ldots, 0\}.\\
				\end{split}
				\eeq
				 Hence the multiplicity of $j$ in $(H_{\lambda}^{i+t} \mod t)$ is one more than its multiplicity in $H_{\lambda}^i$.

\section{Reducing a $t$-core modulo a divisor of $t$} \label{sec:Counting Fibres of Core Maps}

 In this section we enumerate the fibres of the retractions
 \beq
 \core_b: \mc C_a^k \to \mc C_b^k,
 \eeq
when $b$ is a divisor of $a$. In particular, we demonstrate that the size of these fibres  is quasipolynomial in $k$. This could be regarded as a warm-up for our main theorems.

\begin{defn}
A function $f : \mb N \rightarrow \mb \Q$ is a \emph{quasipolynomial}, provided  there exists a positive integer $p$ so that for each $0 \leq i <p$, the function
\beq
n \mapsto f(i+np)
\eeq
where $n \in \W$ is a polynomial. The \emph{degree} of $f$ is the maximum of the degrees of these polynomials. The integer $p$ is a \emph{period} of $f$.
\end{defn}
 
 This definition is equivalent to the one in \cite[Section 4.4]{Stanley}.
 \\
 Let
 \beq
 f_b^a: \Z/a\Z \to \Z/b\Z
 \eeq
 be the usual  remainder mod $b$ map, and let
 \beq
 \rho_b^a=(f_b^a)_*: \mc M_{\fin}(\Z/a \Z) \to \mc M_{\fin}( \Z/b\Z)
 \eeq
be the induced map on multisets. We use the same notation for the restriction
 \beq
 \rho_b^a: \Zak \to \Zbk.
 \eeq

 \begin{prop} \label{warm.up.commute}
The following diagram commutes:
\beq
\begin{tikzcd}
\mc C_{a}^k \arrow{r}{\core_{b}} \arrow{d}{\A_a} & \mc C_{b}^k \arrow{d}{\A_b}
\\
\Zak \arrow{r}{\rho_b^a} & \Zbk
\end{tikzcd}
\eeq
\end{prop}

\begin{proof}
Going down, then right, then up the diagram is the composition
\beq
\begin{split}
\A_b^{-1}\: \rho_b^a \: \A_a &=r  \: \beta^{-1} \: c_b\: \rho_b^a \: \rho_a \: \beta \: u^k \\
&= \beta^{-1}  R  \: c_b \: \rho_b \: \beta  \: u^k \\
&=\core_b. \\
\end{split}
\eeq
(We have used Proposition \ref{core.beta} and Equation \eqref{inv.of.at}.)
\end{proof}

 \begin{lemma}\label{prop: cardinality of ro*}
For $G \in \mc M_{\fin}(\Z/b\Z)$, we have
\beq
\# (\rho^a_b)^{-1}(G)=\prod_{j \in \Zb}  \inlinebnm{\frac{a}{b}}{G(j)}.
\eeq
\end{lemma}

\begin{proof} This follows from Lemma \ref{fibres.for.multisets}.
\end{proof}
 
Given $\sigma \in \mc C_b$, let 
\beq
N_{\sigma}(k) =\# \{\lambda \in \mc C_{a} \mid \core_b \lambda = \sigma, \ell(\lambda) \leq k\}.
\eeq

\begin{thm}\label{warmup.theorem}
There is a quasipolynomial $Q_{\sigma}(k)$ of degree $a-b$ and period $b$, so that for $k \geq \ell(\sigma)$, we have $\Nsk=Q_{\sigma}(k)$. The leading coefficient of $Q_\sigma(k)$ is $\dfrac{1}{(\frac{a}{b}-1)!^{b}}$.
\end{thm}

\begin{proof} Let $c=\frac{a}{b}$.
By Proposition \ref{warm.up.commute} and Lemma~\ref{prop: cardinality of ro*}, for $\ell(\sigma) \leq i < ~\ell(\sigma)+~b$
we have
\begin{align*}
N_{\sigma}(i + nb) & = \prod\limits_{j=0}^{b-1}\inlinebnm{c}{H_{\sigma,b}^{i+nb}(j)}
\\
& = \prod\limits_{j=0}^{b-1} \inlinebnm{c}{H_{\sigma,b}^{i}(j) + n} \\
&= \prod\limits_{j=0}^{b-1}\binom{n+ H_{\sigma,b}^{i}(j) + c-1}{H_{\sigma,b}^{i}(j) + n} \\
&= \prod\limits_{j=0}^{b-1}\binom{n+ H_{\sigma,b}^{i}(j) + c-1}{c-1}.\\
\end{align*}
 Now each
 \beq
 \binom{n+ H_{\sigma,b}^{i}(j) +c-1}{c-1}
 \eeq
 is a polynomial function of $n$ of degree $c-1$ and leading term 
 \beq
 \frac{n^{c-1}}{(c-1)!}.
 \eeq
Therefore $N_{\sigma}(i + nb)$ is a polynomial in $n$ of degree $a-b$ and leading coefficient $\dfrac{1}{(\frac{a}{b}-1)!^{b}}$. Thus the restriction of $N_{\sigma}(k)$ to the coset $i+b \W$ is a polynomial, and the theorem follows.

This shows that for $k \geq \ell(\sigma)$, $N_{\sigma}(k)$ is a quasipolynomial of degree $a-b$ and leading coefficient $\dfrac{1}{(\frac{a}{b}-1)!^{b}}$.  
\end{proof}

\begin{ex}
Let $a=6, b=2$, and $\sigma = (4,3,2,1) \in \mc C_2$. Then $\ell(\sigma) = 4$ and $c = \frac{a}{b}=3.$ We have
$$H_{\sigma}^4 =
\{7,5,3,1\} \text{  and  }H_{\sigma}^5 = \{8,6,4,2,0\} .$$
Hence, $H_{\sigma,2}^4(j) =\begin{cases}
0\text{ for } j =0\\
4\text{ for } j =1
\end{cases}$ and
$H_{\sigma,2}^5(j) =\begin{cases}
5\text{ for } j =0\\
0\text{ for } j =1.
\end{cases}$
\\
By Theorem~\ref{warmup.theorem}, for $n \geq 0$,
\begin{align*}
N_\sigma(4+2n) &= \binom{n+ 2}{2} \binom{n+ 6}{2} \\
&= \frac{1}{4}(n^4 +14n^3 + 65n^2 +112n + 60)
\end{align*}
 and \begin{align*}
 N_\sigma(5+2n) &= \binom{n+ 7}{2} \binom{n+ 2}{2} \\
 &= \frac{1}{4}(n^4 +16n^3 + 83n^2 + 152n + 84). 
 \end{align*}
\end{ex}

\section{Converting to a Multiset Matching Problem}  \label{sec: Converting to a Multiset Matching Problem} 

In this section we recall the map $\core_{s,t}$ from the introduction, and interpret the fibre counting problem in terms of multisets. By the usual Chinese Remainder Theorem, we may view $\Z/m\Z$ as the fibre product of $\Z/s\Z$ and $\Z/t\Z$ over $\Z/d\Z$. 
So we then investigate the effect of the functor $S \leadsto \Sk$ on a fibre product. This study allows us to express ~$\Nstk$ in terms of classical combinatorial constants arising in margin problems for integral matrices. Moreover we  give a factorization of $\Nstk$, which allows a reduction to the case where $s,t$ are relatively prime.

\subsection{The Map $\core_{s,t}$}

Let $s, t \in \mb N$, and put $d=\gcd(s,t)$ and $m=\lcm(s,t)$. Consider the map 
\beq
\core_{s,t}: \mc C_{m} \rightarrow \mc C_s \times \mc C_t
\eeq
taking an $m$-core $\lambda$ to $(\core_s \lambda, \core_t \lambda)$. 
As in Proposition \ref{warm.up.commute}, we have a commutative diagram
\begin{equation}\label{CRT commutative diag}
\begin{tikzcd}
\mc C_{m}^k \arrow{r}{\core_{s,t}} \arrow{d}{\A_m} & \mc C_{s}^k \times \mc C_t^k \arrow{d}{\A_s \times \A_t}
\\
\Zmk \arrow{r}{\rho_{s,t}} & \Zsk \times \Ztk
\end{tikzcd}
\end{equation}
where the maps out of $\mc C_m^k$ and $\mc C_s^k \times \mc C_t^k$ are the bijections of Proposition~\ref{prop: cores multisets bijection}, and $$\rho_{s,t}= \rho^m_s \times \rho^m_t.$$



Let $\Nstk$ be the cardinality of the fibre of $\core_{s,t}$ over $(\sigma, \tau)$ for $k \in \mb N$. Thus
 $$\Nstk = \# \{\lambda \in \mc C_{m}^k \mid \core_s \lambda = \sigma, \core_t \lambda = \tau\}.$$ 
By the commutativity of \eqref{CRT commutative diag}, counting fibres of $\core_{s,t}$ is equivalent to counting fibres of $\rho_{s,t}$. 

\subsection{Matchings}


For finite sets $S,T$, consider the projection maps $\pr_S: S \times T \to S$ and $\pr_T: S \times T \to T$.
There are corresponding multiset maps
$$(\pr_S)_* : \STk \to \Sk, \hspace{0.2 cm} \hspace{0.2 cm} (\pr_T)_* : \STk \to \Tk. $$

and $$\pr: \STk \to \Sk \times \Tk$$  given by $\pr= (\pr_S)_* \times (\pr_T)_*.$

\begin{defn}
Let $F \in \Sk$ and $G \in \Tk$.  We say that $\Phi \in \STk$ is a \emph{matching} from $F$ to $G$, provided  $\pr(\Phi) = (F, G)$.
\end{defn}

Say $|S|=m$ and $|T|=n$, with $S=\{x_1, \ldots, x_m\}$ and $T=\{y_1, \ldots, y_n\}$, 
Given $F,G$ as above, define vectors
\beq
\vec F=(F(x_1), \ldots, F(x_m))
\eeq
 and 
 \beq
 \vec G=(G(y_1), \ldots, G(y_n)).
 \eeq
So $k$ is the sum of the components of $\vec F$, and also  the sum of the components of   $\vec G$.

One says that an $m \times n$ matrix $A$ has \emph{row margins} $\vec F$, if $F(x_i)$ is the sum of the entries of the $i$th row for each $i$. Similarly $A$ has \emph{column margins} $\vec G$, if $G(y_i)$ is the sum of the entries of the $j$th column for each ~$j$.

\begin{prop} \label{prop: matching matrix bijection}  Given $F \in \inlinebnm{S}{k}$ and $G \in \inlinebnm{T}{k}$, there is a bijection between the set of matchings from $F$ to $G$ and the set of non-negative integral matrices with row margins $\vec F$ and column margins $\vec G$.
\end{prop}

\begin{proof} Suppose that $(m_{ij})$ is such a matrix.  Then 
\beq
 \Phi(x_i, y_j)=m_{ij}
 \eeq
  is a matching from $F$ to $G$, and this gives the required bijection.
\end{proof}

Write $M_{F,G}$ for the number of matrices with nonnegative integer entries having row margins $\vec F$ and column margins $\vec G$.
According to \cite[Corollary 8.1.4]{brualdi2006combinatorial}, if the sum of the components of $\vec F$ equals the sum of the components of $\vec G$, then $M_{F,G} \geq 1$.

\begin{cor}\label{card.pr.here}
The cardinality of the fibre of $\pr$ over $(F,G)$ is $M_{F, G}$. In particular,  $\pr$ is surjective.
\end{cor}

\subsection{Coprime Case}
\label{sec: Counting fibres of the Sun Tzu map}

In this subsection let $s,t$ be relatively prime, and set $m=st$.  Put $S=\Z/s\Z$, $T=\Z/t\Z$, and $M=\Z/m\Z$. The Chinese Remainder Theorem  gives a bijection 
\beq
\ms S: M \overset{\sim}{\to} S \times T
\eeq
and for each $n \geq 0$, the map $\rho_{s,t}$ is the composition
\beq
\multi{M}{n} \overset{\ms S_*}{\to} \multi{S \times T}{n} \overset{\pr}{\to} \multi{S}{n} \times \multi{T}{n}.
\eeq

Now let $\sigma \in \mc C_s$, and $\tau \in \mc C_t$. Put $\ell_0 = \max(\ell(\sigma), \ell(\tau))$ and fix  $i \geq \ell_0$.
 For each $k \geq 0$, the bijection
\beq
\ms A_s: \mc C_s^{i+mk} \overset{\sim}{\to} \multi{S}{i+mk}
\eeq
of Proposition \ref{prop: cores multisets bijection} maps $\sigma$ to 
\beq
F^k:=H_{\sigma,s}^{i+mk},
\eeq
with  notation as in \eqref{defn.of.H.thing}. Similarly $\ms A_t: \mc C_t^{i+mk} \overset{\sim}{\to} \multi{T}{i+mk}$ maps $\tau$ to
\beq
G^k:=H_{\tau,t}^{i+mk}.
\eeq

Note that $\ms S_*$ is a bijection. By Corollary \ref{card.pr.here}, we have
\beq
N_{\sigma,\tau}(i+mk) = M_{F^k,G^k}.
\eeq

Moreover by Lemma~\ref{lemma: periodicity of betaset mod s set} we know 
$$F^k(j) = H_{\sigma, s}^i(j) + tk \hspace{0.1 cm} \text{   and   } \hspace{0.1 cm} G^k(j) = H_{\tau,t}^i(j) + sk.$$

Putting all this together:

\begin{thm} \label{rel.prime.cor.today} When $s,t$ are relatively prime, then for $i \geq \ell_0$, we have
\beq
N_{\sigma,\tau}(i+stk)=M_{F^k,G^k},
\eeq
where 
$$\vec F^k =(a_{0i}, a_{1i}, \ldots, a_{(s-1)i})   +kt(\underbrace{1,1,\ldots, 1}_{\text{$s$ times}})$$
and 
$$\vec G^k = (b_{0i}, b_{1i}, \ldots, b_{(t-1)i})+ks(\underbrace{1,1,\ldots, 1}_{\text{$t$ times}}),$$
with  $a_{ji} = H_{\sigma, s}^i(j)$ and $b_{ji}=H_{\tau,t}^i(j)$.
 \end{thm}
 

\subsection{A Factorization in the Noncoprime Case}

In this subsection we ``factor''  $N_{\sigma, \tau}$ into products of $N_{\sigma',\tau'}$ with the sizes of $\sigma'$ and $\tau'$ relatively prime.

Given sets $S, T, D$ and maps $f: S \to D$ and $g: T \to D$, we have the commutative diagram
\begin{equation}\label{comm diag: fibre prod}
\begin{tikzcd}
S \times_D T \arrow{r}{g'} \arrow[swap]{d}{f'} & T  \arrow{d}{g} 
\\
S \arrow{r}{f} & D
\end{tikzcd}
\end{equation}
where $$S \times_D T = \{(a,b) \in S \times T \mid f(a)=g(b)\}$$ is the fibre product of $f$ and $g$, and $f', g'$ are projections to $S$ and $T$. By applying the functor $S \leadsto \inlinebnm{S}{k}$ to \eqref{comm diag: fibre prod} we get another commutative diagram
\begin{center}
\begin{tikzcd}
  \SDTk
  \arrow[drr, bend left, "(g')_*"]
  \arrow[ddr, bend right, "(f')_*"]
  \arrow[dr, dotted, "{\epsilon}" description] & & \\
    & \SkDkTk \arrow[r, "(g_*)'"] \arrow[d, "(f_*)'"]
      & \Tk \arrow[d, "g_*"] \\
& \Sk \arrow[r, "f_*"] &\Dk
\end{tikzcd}
\end{center}
where 
\beq
\epsilon:  \SDTk \to \SkDkTk
\eeq
 is defined by
$$\epsilon(\Phi) = ((\pr_S)_*(\Phi), (\pr_T)_*(\Phi)).$$
Again, the maps $(f_*)'$ and $(g_*)'$ out of the fibre product $\SkDkTk$ are the projections.

Let $(F, G) \in \Sk \times \Tk$ such that $f_*(F) = g_*(G)$. For $j \in D$, let $S_j$ and $T_j$ be the fibres of $f$ and $g$ over $j$ respectively. Write $F_j$ for the restriction of $F$ to $S_j$, and $G_j$ for the restriction of $G$ to $T_j$.

\begin{prop}    \label{corollary: cardinality of epsilon}        
The cardinality of the fibre of $\epsilon$ over $(F,G)$ is $\prod\limits_{j \in D} M_{F_j, G_j}$.
In particular,  $\epsilon$ is surjective.
 \end{prop}

\begin{proof}
 Let $k_j=|F_j|$.  Then 
$$k_j = \sum\limits_{s \in S_j}F(s) = f_*(F)(j) = g_*(G)(j) = \sum\limits_{t \in T_j}G(t) = |G_j|.$$ 
For each $j \in D$, we have a surjective map 
$$\pr_j: \SdTdk \to \Sdk \times \Tdk$$
as before.

For each $j \in D$, pick a multiset $\Phi_j$ in the fibre of $\pr_j$ over $(F_j, G_j)$. This can be done in $M_{F_j,  G_j}$ ways (Corollary \ref{card.pr.here}). Now define the multiset $\Phi \in \SDTk$ as
$$\Phi(s,t) = \Phi_j(s,t) \text{\hspace{0.1 cm} if \hspace{0.1cm}} (s,t) \in S_j \times T_j.$$
The cardinality of $\Phi$ is
$$|\Phi| = \sum\limits_{j \in D}\sum\limits_{(s,t) \in S_j \times T_j}|\Phi_j(s,t)| = \sum\limits_{j \in D} k_j = k,$$ as required. From the construction of $\Phi$ it is clear that $\epsilon(\Phi) = (F, G)$.
\end{proof}


\begin{theorem}\label{prop: Nst = prodNsigmajtauj}
Let $s,t \in \mb N$, and put $\gcd(s,t) =d$ and $\lcm(s,t) =m$. Suppose $\sigma \in \mc C_s$ and $\tau \in \mc C_t$ with $\core_d(\sigma)=\core_d(\tau)$.
 Then for $i \geq \ell_0$ and $0 \leq j < d$, there exist $\frac{s}{d}$-cores $\sigma_j^i$, $\frac{t}{d}$-cores $\tau_j^i$ and nonnegative integers $\ell_j^i$ such that for all $k \geq 0$ we have
\beq
 N_{\sigma, \tau}(i +mk)= \prod_{j=0}^{d-1} N_{\sigma_{j}^i,\tau_{j}^i} \left(\ell_{j}^i + \frac{m}{d}k \right).
\eeq
\end{theorem}

\begin{proof}
As above, we write $(H_{\sigma,s}^{i+mk})_j$ for the restriction of the multiset $H_{\sigma,s}^{i+mk}$ to 
\beq
(\Zs)_j= \{ x \in \Zs \mid x \equiv j \mod d \}.
\eeq
 Put $F_j^k=(H_{\sigma,s}^{i+mk})_j$ and $G_j^k=(H_{\tau,t}^{i+mk})_j$. Then by the commutative diagram \eqref{CRT commutative diag} and Proposition \ref{corollary: cardinality of epsilon}, we have
\beq
N_{\sigma, \tau}(i+mk) = \prod\limits_{j=0}^{d-1}M_{F_j^k, G_j^k}
\eeq
for $i \geq \ell_0$.
Via Proposition \ref{prop: cores multisets bijection}, define the $\frac{s}{d}$-core $\sigma_j^i$ so that
$$H_{\sigma_j^i, \frac{s}{d}}^{l_j^i} = F_j^0,$$
and the $\frac{t}{d}$-core $\tau_j^i$ by  
$$H_{\tau_j^i, \frac{t}{d}}^{l_j^i} = G_j^0$$
where $l_j^i = |F_j^0| = |G_j^0|.$ Then again by the commutative diagram and Corollary \ref{corollary: cardinality of epsilon} in the relatively prime case, $$N_{\sigma_{j}^i,\tau_{j}^i}\left(\ell_{j}^i + \frac{m}{d}k \right) = M_{F_j^k, G_j^k}.$$ This completes the proof.
\end{proof}

 \begin{ex} \label{example1: noncoprime}  
Let $s =4, t =6$ and $\sigma = (3,1,1), \tau = (3,2)$. Then $d = 2, m =12$, $H_{\sigma} = \{5,2,1\}$, $H_{\tau} = \{4,2\}$ and $\ell_0 =3$. Let $F_j^k$ and $G_j^k$ be the multisets as in the theorem above.
\bigskip

For $i=12$,
\begin{align*}
F_0^k = (3k+3, 3k+4), && G_0^k = (2k+3, 2k+3, 2k+1), &&\\
F_1^k = (3k+2,3k+3), && G_1^k = (2k+2, 2k+2,2k+1),&&\\
\sigma_0^{12} = (1), && \tau_0^{12} =(1,1), && \ell_0^{12}= 7,\\
\sigma_1^{12} = (1), && \tau_1^{12} =\emptyset, && \ell_1^{12}=5.
\end{align*}
Thus $$N_{\sigma, \tau}(12 + 12k) = N_{(1),(1,1)}(7+6k) \cdot N_{(1),\emptyset}(5+6k).$$
We will continue with this in Example \ref{ex2:noncoprime case}.
\end{ex}

\section{Preliminaries for Polytopes}
\label{sec: Preliminaries of polytopes and notation}

We now review notions concerning polytopes which we will need. A suitable reference is \cite[Section 4.6.2]{Stanley}.
 
 \subsection{Basic Terminology}
 
Given an $m \times n$ matrix $A$ and a vector $\vec b \in \mb R^m$, we define the polyhedron
$$\mc P(A, \vec b) = \{\vec x \in \mb R^n \mid A\vec x \leq \vec b, \vec x \geq 0\}.$$
Following convention, `$\vec v_1 \leq \vec v_2$' means that each component of $\vec v_1$  is less than or equal to the corresponding component of $\vec v_2$.

\begin{defn} 
We say that the pair $(A, \vec b)$ is \emph{of bounded type}, provided $\mc P(A, \vec b)$ is bounded.  A bounded polyhedron is called a \emph{polytope}.
\end{defn}
\begin{defn}
The \emph{dimension} of a polytope, $\dim(\mc P)$, is the dimension of the affine space spanned by $\mc P$:
\beq
\ASpan(\mc P) = \{\vec x + \lambda(\vec y- \vec x) \mid \vec x, \vec y \in \mc P,\lambda \in \mb R\}
\eeq
  When $\dim(\mc P) = d$, we call it a $d$-polytope. 
\end{defn}

Write $\conv(\{\vec v_1, \vec v_2, \ldots, \vec v_m\})$ for the convex hull of $\{\vec v_1, \vec v_2, \ldots, \vec v_m\} \subset$~$\mb R^n$.

\begin{defn}
A  \emph{convex polytope} $\mc P$ in $\mb R^n$ is the convex hull of an affine independent set $\{\vec v_1, \vec v_2, \ldots, \vec v_m\} \subset \mb R^n$, called the \emph{vertices} of $\mc P$.
If all vertices of $\mc P$ have integer coordinates,  then it is called a \emph{lattice polytope}.
\end{defn}

\begin{lemma}
\label{lem: P(A,b)empty}
Let $A$ be an $m \times n$ matrix and $\vec b, \vec c \in \R^m$. If
$\mc P(A, \vec b)=\emptyset$,  then $\mc P(A,\vec bk+ \vec c)=\emptyset$ for $k \gg 0$.
\end{lemma}

\begin{proof}
For a matrix $A$ and a vector $\vec b$, the inequality $A \vec x \leq \vec b$ has a solution for $\vec x$, if and only if $\vec y \cdot \vec b \geq 0$ for each row vector $\vec y \geq \vec 0$ with $\vec yA = \vec 0$ \cite[Corollary 7.1 e, Farkas' Lemma (variant)]{schrijver1998theory}. By the hypothesis, there exists such a $\vec y$ with $\lnot (\vec y \cdot \vec b \geq 0)$. Therefore for some $i$, the $i$th component $y_i$ of $\vec y$ is positive, and the $i$th component $b_i$ of $\vec b$ is negative. But then for $k$ large 
\beq
y_i(b_ik+c_i)<0,
\eeq
where $c_i$ is the $i$th component of $\vec c$.
Therefore $\mc P(A,\vec bk+ \vec c)=\emptyset$.
\end{proof}

\begin{lemma} \label{lem: P(A,b) bounded}
Let $A$ be an $m \times n$  matrix and $\vec b, \vec b' \in \R^m$. Suppose $\mc P(A, \vec b) \neq \emptyset$. Then, $\mc P(A,\vec b)$ is bounded iff $\mc P(A,\vec b')$ is  bounded.
\end{lemma}

\begin{proof}
The \emph{characteristic cone} of a nonempty polytope $\mc P$ is defined as
$$
\Char \cone \mc P = \{\vec y \mid \vec x+\vec y \in \mc P\hspace{0.1 cm} \text{for all}\hspace{0.1 cm} \vec x \in \mc P\} = \{\vec y \mid A\vec y \leq \vec 0\}.
$$
The nonempty polytope $\mc P$ is bounded if and only if $\Char \cone \mc P = \{0\}$ \cite[page 100, (5)]{schrijver1998theory}. The $\Char \cone \mc P$ does not depend on $\vec b$ and hence the boundedness of the polytope $\mc P(A, \vec b)$ does not depend on $\vec b$.
\end{proof}

\begin{defn}
A matrix $A$ is said to be \emph{totally unimodular}, provided the determinant of each square submatrix of $A$ is $+1$, $-1$ or $0$.
\end{defn}

\begin{prop}[{\cite[Corollary~19.2a, Hoffman and Kruskal's Theorem]{schrijver1998theory}}]
\label{prop: Scrijver 1}
Let $A$ be an integral matrix. Then $A$ is totally unimodular if and only if for each integral vector $\vec b$ the polyhedron $\{\vec x \mid A\vec x \leq \vec b,  \vec x \geq 0\}$ is integral.
\end{prop}

 \subsection{Relative Volume}
When $\mc P \subset \mb R^n$ is a lattice polytope, not necessarily of dimension ~$n$, there is yet a ``relative volume'' of $\mc P$, which 
we recall from \cite[Section 4.6]{Stanley}. Let $d=\dim \mc P$.

Since $\mc P$ is a lattice polytope, the intersection of $\ASpan(\mc P)$ with $\Z^n$ is a translation of a free abelian group of rank $d$.  Therefore there is an affine isomorphism
\beq
T: \ASpan(\mc P)  \xrightarrow{\sim}   \mb R^d
\eeq
with 
\beq
T(\ASpan(\mc P) \cap \Z^n) \xrightarrow{\sim}  \Z^d.
\eeq
The \emph{relative volume} of $\mc P$ is the volume of $T(\mc P)$, and is independent of the choice of $T$. Of course when $\dim(\mc P) = n$, the relative volume agrees with the
usual volume.

\subsection{Transportation Polytopes}
Let $\vec r =(r_1, r_2, \ldots, r_s)$ and $\vec c  =(c_1, c_2, \ldots, c_t)$ be vectors whose components are nonnegative integers, and with $\sum r_i=\sum c_j$.  The nonnegative real matrices with row sum $r_i$ and column sum $c_j$ form a polytope $\mc M(\vec r, \vec c)$ called the \emph{transportation polytope} for margins $\vec r$ and ~$\vec c$. According to  \cite[Theorem 8.1.1]{brualdi2006combinatorial}, this polytope has dimension $(s-1)(t-1)$. By \cite[Section 2]{jurkat},  $\mc M(\vec r, \vec c)$ is a lattice polytope. Write $\M_{s,t}$ for the set of $s \times t$ real matrices.
Let $\pi : \M_{s,t} \rightarrow \M_{s-1, t-1}$ be the map defined by  omitting the last row and column. We put
\beq
\mc M'(\vec r,\vec c)=\pi( \mc M(\vec r,\vec c)).
\eeq
 
\begin{prop}  \label{transport.A.b.form}
The polytope $\mc M'(\vec r,\vec c)$ has dimension $(s-1)(t-1)$.  It takes the form $P(A,\vec b)$, where $A$ is a totally unimodular matrix, and $(A,\vec b)$ is of bounded type.
The integer points of  $\mc M(\vec r,\vec c)$ are mapped bijectively by $\pi$ onto the integer points of $\mc M'(\vec r,\vec c)$. The volume of $\mc M'(\vec r,\vec c)$ is the relative volume of 
$\mc M(\vec r,\vec c)$. 
\end{prop}

\begin{proof}
The polytope $\mc M'(\vec r,\vec c) \subset \M_{s-1, t-1}$ comprises the nonnegative solutions to the following $s+t-1$ constraints:
\beq
\begin{split} 
\sum\limits_{j=1}^{t-1}x_{ij} &\leq r_i \: \: \: \text{ for } 1 \leq i \leq s-1\\
\sum\limits_{i=1}^{s-1}x_{ij} &\leq c_j \: \: \: \text{ for } 1 \leq j \leq t-1\\
- \sum\limits_{i=1}^{s-1}\sum\limits_{j=1}^{t-1}x_{ij}  &\leq  r_s - \sum\limits_{j=1}^{t-1}c_j. \\
\end{split}
\eeq

In particular, we may write
\beq
\vec b=\left(r_1,\ldots, r_{s-1},c_1, \ldots, c_{t-1}, b_{s+t-1} \right)^t,
\eeq
where
\beq
b_{s+t-1}= r_s-\sum_{j=1}^{t-1} c_j=c_t-\sum_{i=1}^{s-1} r_i.
\eeq
If $A$ is the evident $(s+t-1) \times (s-1)(t-1)$ coefficient matrix, then $\mc M'(\vec r,\vec c)=\mc P(A,\vec b)$.

Conversely, given an integral vector $\vec b_*=(b_1, \ldots, b_{s+t-1})^t \in \Z^{s+t-1}$, the polytope $\mc P(A,\vec b_*)=\pi(\mc M(\vec r_*,\vec c_*))$, with
\beq
\vec r_*=(b_1, \ldots, b_{s-1}, \sum_{i=0}^{t-1}b_{s+i})^t \text{ and } \vec c_*=(b_s, b_{s+1}, \ldots, b_{s+t-2}, b_{s+t-1}+\sum_{i=1}^{s-1} b_i)^t.
\eeq

Again  by \cite[Section 2]{jurkat},  $\mc M(\vec r_*, \vec c_*)$ is a lattice polytope, when  $\vec r_*,\vec c_*$ are nonnegative. (Otherwise it is empty, still technically a lattice polytope.)
Therefore the projection $\mc P(A, \vec b_*)$ is also a lattice polytope for each integral $\vec b_*$. Proposition \ref{prop: Scrijver 1} now implies  that the matrix $A$ is totally unimodular.

Next, write $\mc A \mc M(\vec r, \vec c)$ for the set of real $s \times t$ matrices with margins $\vec r$ and $\vec c$.  One checks that
 \beq 
 \ASpan(\mc M(\vec r, \vec c))=\mc A \mc M(\vec r, \vec c),
\eeq
and moreover that the restriction
 \beq
T:  \ASpan (\mc M(\vec{r}, \vec{c})) \xrightarrow{\sim}  \M_{s-1,t-1}
\eeq
 of $\pi$ is an affine isomorphism, giving a bijection on integer points. This gives the dimension and relative volume assertions.   
 Boundedness is clear. \end{proof}


\begin{ex}        \label{Ex.V23}  

Let $\mc P$ be the transportation polytope for row margins $(2,2,2)$ and column margins $(3,3)$. Then $\ASpan(\mc P)$ is the affine space of matrices of the form

\begin{equation} \label{v23}
\begin{pmatrix}

x & 2-x \\

y & 2-y \\

3-x-y & x+y-1 \\

\end{pmatrix},
\end{equation}
for $x,y \in \R$, and $\mc P$ is the subset of $\ASpan(\mc P)$ defined by the constraints $1 \leq x+y \leq 3$ and $0 \leq x,y \leq 2$.  

The map $T$ taking \eqref{v23} to $(x,y)$ is an affine isomorphism to $\R^2$, taking integer points to integer points. Then the relative volume of $\mc P$, i.e., the volume of $T(\mc P)$, is the area of the region in Figure 1, which is ~$3$.
\begin{figure}[!h]\label{Figure 1}
\begin{tikzpicture}[scale=2]
\draw (1,0) -- (2,0) -- (2,1) -- (1,2) -- (0,2) -- (0,1)--(1,0);
\node[label=left:{$(1,0)$}] at (1,0) {};
\node[label=right:{$(2,0)$}] at (2,0) {};
\node[label=right:{$(2,1)$}] at (2,1) {};
\node[label=right:{$(1,2)$}] at (1,2) {};
\node[label=left:{$(0,2)$}] at (0,2) {};
\node[label=left:{$(0,1)$}] at (0,1) {};
\end{tikzpicture}
\caption{$T(\mc P)$}
\end{figure}


\end{ex}

\begin{remark} It is notoriously difficult to compute such volumes in general. The transportation polytope for $n \times n$ matrices with row and column margins  $(1, \ldots, 1)$ is the famous Birkhoff polytope. Finding its volume is an open problem; see for instance \cite{de2009generating}.
\end{remark}

\section{Counting Integer Points in Lattice Polytopes}
\label{sec: Counting integer points in polytopes}

In this section, we adapt the Ehrhart theory of counting integer points in lattice polytopes to accommodate our  families of transportation polytopes.

\begin{thm} (Ehrhart, see {\cite[Corollary~4.6.11]{Stanley}})
\label{thm: Ehrhart}
Let $\mc P$ be a lattice $d$-polytope in $\mb R^n$.  Then the number of integer points in the polytope 
\beq
k\mc P = \{k \vec \alpha \mid \vec \alpha \in \mc P\},
\eeq
with $k$ a positive integer, is a polynomial in $k$ of degree $d$, with leading coefficient equal to the relative volume of $\mc P$.
\end{thm}

\bigskip

For $(A, \vec b)$ of bounded type, write $N(A, \vec b)$ for the number of integer points in the polytope $\mc P(A, \vec b)$.

\begin{prop}
\label{Ehrhart for totunimod}
Let $A$ be an $m \times n$ totally unimodular matrix and $\vec b \in \mb Z^m$. Suppose $(A, \vec b)$ is of bounded type and that $\mc P(A,\vec b) \neq \emptyset$. 
Then there is a polynomial $f(k)$ so that for positive integers $k$, we have $N(A, \vec bk)=f(k)$. Moreover $\deg f=\dim(\mc P(A,\vec b))$ and the leading coefficient of $f$ is the relative volume of $\mc P(A,\vec{b})$.
 
\end{prop}
\begin{proof}
By Proposition~\ref{prop: Scrijver 1}, $\mc P(A,\vec{b}k)$ is a lattice polytope. Moreover, $\mc P(A,\vec{b}k) = k \mc P(A, \vec{b})$. Therefore the conclusion follows by Ehrhart's Theorem. %
\end{proof}

\begin{lemma}
\label{rowoperation totunimod lemma}
Let $A$ be an $m \times n$ totally unimodular matrix with $a_{11} \neq 0$.   Let $A'$ be the  matrix obtained by applying all the row operations 
\beq
R_i \mapsto R_i- a_{i1}a_{11}^{-1}R_1,
\eeq
 for $2 \leq i \leq m$.   Then $A'$ is totally unimodular.
\end{lemma}
This seems to be well-known, but we provide a proof for the convenience of the reader.
\begin{proof} 
 
  Let $I \subseteq \{1,2, \ldots, m\}$ and  $J \subseteq \{1,2, \ldots, n\}$, with $|I|=|J|>0$. Let $A_{IJ}$ denote the submatrix of $A$ that corresponds to the rows with index in $I$ and columns with index in $J$. We need to show that $\det(A'_{IJ})$ is $1, -1$ or $0$. 

Case 1: If $1 \in I$, $\det(A'_{IJ})= \det(A_{IJ}) = 1,-1$ or $0$, since $A$ is totally unimodular and the determinant remains unchanged under such a row operation.

Case 2: If $1 \notin I$, $1 \in J$, $\det(A'_{IJ}) = 0$ since the first column of $A'_{IJ}$ is 0.

Case 3: If $1 \notin I$, $1 \notin J$, let $\tilde{I} = I \cup \{1\}$ and $\tilde{J} = J \cup \{1\}$. Then $\det(A'_{\tilde{I}\tilde{J}}) = a_{11} \det(A'_{IJ})$. By Case 1, we see $\det(A'_{\tilde{I}\tilde{J}})$ is $1$, $-1$ or $0$.

Hence $A'$ is totally unimodular.   \end{proof}

\begin{lemma} \label{zero.row.lemma}
Let $A$ be an $m \times n$ matrix, and $\vec b,\vec c \in \R^m$. Put
\beq
Z=\{ 1 \leq i \leq m \mid \text{ the $i$th row of $A$ is $0$}\},
\eeq
and $m_*=m-|Z|$. Let $A_*$ be the $m_* \times n$ matrix obtained by deleting the zero rows from $A$.
Similarly form $\vec b_*,\vec {c_*} \in \R^{m_*}$  by deleting the corresponding components from $\vec b$ and $\vec c$.

Then  one of the following must hold:
\begin{enumerate}
\item \label{one.zero.lem} $\mc P(A,\vec bk+\vec c)=\emptyset$ for $k \gg 0$.
\item \label{two.zero.lem} $\mc P(A,\vec bk+\vec c)=\mc P(A_*,\vec b_*k+\vec {c_*})$ for $k \gg 0$.
\end{enumerate} 
 \end{lemma}
 
 \begin{proof} 
 If there exists $i \in Z$ with $b_i<0$, then $\mc P(A,\vec b)=\emptyset$, so \eqref{one.zero.lem} holds by Lemma \ref{lem: P(A,b)empty}.
 So we may assume that $b_i \geq 0$ for all $i \in Z$. If there exists $i \in Z$ so that both $c_i <0$ and $b_i=0$, then $\mc P(A,\vec bk+\vec c)=\emptyset$ for all $k \geq 0$.
 Otherwise, the zero rows of $A$ correspond to inequalities $0 \leq b_ik+c_i$ with either $b_i>0$, or $b_i=0$ and $c_i \geq 0$. For large $k$, these inequalities will hold, giving  \eqref{two.zero.lem}.
 \end{proof}
 

\begin{thm}
\label{totunimod polytope polynomial}
Let $A$ be an $m \times n$ totally unimodular matrix and $\vec b, \vec c \in \mb Z^m$. Suppose $\mc P(A, \vec b)$ is  bounded. Then there is a polynomial $f(k)$ with $\deg f \leq n$, such that for integers $k \gg 0$, we have 
\beq
N(A, \vec bk+\vec c)=f(k).
\eeq
 If $\dim \mc P(A,\vec b) =n$, and none of the rows of $A$ are 0, then $\deg f = n$ and the leading coefficient of $f$ is the  volume of $\mc P(A,\vec b)$.
\end{thm}

\begin{proof}
By Lemma \ref{lem: P(A,b)empty}, we may assume $\mc P(A,\vec b) \neq \emptyset$.

Suppose first that $A$ has at least one zero row. By Lemma  \ref{zero.row.lemma}, either $N(A, \vec bk+\vec c)=0$ for $k \gg 0$, or $\mc P(A,\vec bk+\vec c)=\mc P(A_*,\vec b_*k+\vec {c_*})$ for $k \gg 0$. Since  $\mc P(A,\vec b)=\mc P(A_*,\vec b_*)$, we may assume  that none of the rows of $A$ are $0$ for the first statement of the theorem.

 Let  $\vec b = (b_1, b_2, \ldots, b_m)^t$ and $\vec c = (c_1, c_2, \ldots, c_m)^t$. 

We proceed by induction on $n$. 
For $n=1$, the matrix $A=(a_1, \ldots, a_m)^t$ is a single column, with each $a_i=\pm 1$. 
Moreover
\beq
\mc P(A,\vec b) =\{ x \in \R \mid a_i x \leq b_i \: \forall i, x \geq 0 \}.
\eeq
Put $I=\{ i \mid a_i=1\}$ and   $J= \{0\} \cup \{ j \mid  a_{j} = -1\}$. Note that
\beq
\begin{split}
\mc P(A,\vec b) \text{ is bounded} & \Leftrightarrow I \neq \emptyset, \: \: \:  \text{ and } \\
 \mc P(A, \vec b) \neq \emptyset & \Leftrightarrow -b_j \leq b_i \: \forall i \in I, j \in J. \\
 \end{split}
\eeq
Let $B^-= \max \{-b_j \mid j \in J\}$, and $B^+= \min\{b_i \mid i \in I\}$.  Since $\mc P(A,\vec b)$ is bounded and nonempty, we have $0 \leq B^- \leq B^+$, and may write
\beq
\mc P(A,\vec b)=[B^-,B^+].
\eeq
Therefore 
\beq
\dim \mc P(A,\vec b)=1 \Leftrightarrow B^- < B^+.
\eeq
Now let $C^-= \max\{-c_j \mid j \in J, \: -b_j= B^-\}$ and
$C^+ = \min\{c_i \mid i \in I, \: b_i=B^+\}$. If $\dim \mc P(A,b)=1$, then for $k \gg 0$, we have
\beq
\mc P(A,\vec bk + \vec c)=[kB^-+C^-,kB^++C^+],
\eeq
so that
\beq
N(A, \vec bk+\vec c) = (B^+ - B^-)k + (C^+ - C^-) + 1.
\eeq
On the other hand, if $\dim \mc P(A,\vec b) =0$, it is easy to see that
\beq
N(A, \vec bk+\vec c) = \begin{cases}
C^+ - C^- +1 &\text{ if } C^+ >C^- \\
0 &\text{ if }  C^+ \leq C^- ,
\end{cases}
\eeq
for $k \gg 0$.
From these calculations we deduce  the theorem when $n=1$.
 
For $n>1$, we further induct on the number of nonzero components of $\vec c$.  Let $A=(a_{ij})$. If $\vec c=\vec 0$, then the conclusion follows from Proposition ~\ref{Ehrhart for totunimod}.  
So suppose $\vec c \neq \vec 0$. By permuting the rows, we may assume $c_1 \neq 0$. 

First, we consider the case where $c_1 > 0$.  We partition  the integer points of $\mc P(A,\vec bk+\vec c)$ as follows. Consider the systems of inequalities:
\begin{equation}
\label{eq: 1}
\begin{array}{ccccccccc}
a_{11}x_1 & + & a_{12}x_2 & + & \ldots & + & a_{1n}x_n & \leq & b_1k\\
a_{21}x_1 & + & a_{22}x_2 & + & \ldots & + & a_{2n}x_n & \leq & b_2k + c_2\\
\vdots &  & \vdots &  & \ldots &  & \vdots &  & \vdots\\
a_{m1}x_1 & + & a_{m2}x_2 & + & \ldots & + & a_{mn}x_n & \leq & b_mk + c_m\\
\end{array}
\end{equation}
and 
\begin{equation}
\label{eq: 2}
\begin{array}{ccccccccc}
a_{11}x_1 & + & a_{12}x_2 & + & \ldots & + & a_{1n}x_n & = & b_1k + \ell\\
a_{21}x_1 & + & a_{22}x_2 & + & \ldots & + & a_{2n}x_n & \leq & b_2k + c_2\\
\vdots &  & \vdots &  & \ldots &  & \vdots &  & \vdots\\
a_{m1}x_1 & + & a_{m2}x_2 & + & \ldots & + & a_{mn}x_n & \leq & b_mk + c_m\\
\end{array}
\end{equation}
for $\ell = 1,2, \ldots, c_1$.  Write $\mc P^0(k)$ for the nonnegative solutions to \eqref{eq: 1} and $\mc P^\ell(k)$ for the nonnegative solutions to \eqref{eq: 2} for $1 \leq \ell \leq c_1$.  Then we have a disjoint union
\begin{equation} \label{peterd}
\mc  P(A,\vec bk+\vec c) \cap \Z^n= \coprod_{\ell=0}^{c_1} (\mc P^\ell(k) \cap \Z^n).
\end{equation}

 The first row of $A$ is nonzero, and we  assume for simplicity that $a_{11} \neq 0$. Solving the equality in \eqref{eq: 2} for $x_1$ gives:
$$x_1 = a_{11}^{-1}\left(b_1k + \ell - \sum\limits_{j \neq 1}a_{1j}x_j \right).$$
Substituting this into the rest  of \eqref{eq: 2} gives:

\beq
\begin{array}{cccccc}
 
a_{21}a_{11}^{-1} \left(b_1k + \ell - \sum\limits_{j\neq 1}a_{1j}x_j \right) & + & \ldots & + & a_{2n}x_n &\leq b_2k + c_2 \\
\vdots &  & \vdots &  & \vdots & \vdots \\
a_{m1}a_{11}^{-1} \left(b_1k + \ell - \sum\limits_{j\neq 1}a_{1j}x_j \right) & + & \ldots & + & a_{mn}x_n & \leq b_mk + c_m. \\
\end{array}
\eeq
for $\ell = 1,2 \ldots, c_1$.
To these inequalities we add 
\beq
 \sum\limits_{j \neq 1} a_{11}^{-1} a_{1j}x_j \leq a_{11}^{-1}(b_1k+ \ell),
\eeq
corresponding to the condition $x_1 \geq 0$.

Write $A'$ for the $m \times (n-1)$ matrix obtained by first applying all the row operations $R_i \mapsto R_i -a_{i1}a_{11}^{-1}R_1$ to $A$ for $2 \leq i \leq m$, removing the first column, and multiplying the first row by $a_{11}^{-1}$.
Then $A'$ is totally unimodular by Lemma \ref{rowoperation totunimod lemma}.  
Define $\vec b' \in \Z^{m}$ by
\beq
\vec b' =(a_{11}^{-1} b_1, b_2-a_{21}a_{11}^{-1} b_1, \cdots, b_m-a_{m1}a_{11}^{-1} b_1)^t,
\eeq
and $\vec c_{(\ell)} \in \Z^{m}$ by
\beq
\vec {c}_{(\ell)}=(a_{11}^{-1} \ell, c_2-a_{21} a_{11}^{-1} \ell, \ldots, c_m-a_{m1} a_{11}^{-1} \ell)^t.
\eeq
Eliminating the first component gives a projection from $\R^n$ to $\R^{n-1}$. This projection maps $\mc P^\ell(k)$  bijectively to $\mc P(A',\vec b'k+ \vec c_{(\ell)})$ and also gives a bijection on integer points. 

Now $\mc P(A, \vec bk + \vec c)$ is bounded by Lemma \ref{lem: P(A,b) bounded}. Therefore the closed subset $\mc P^\ell(k)$ is compact, and it follows that its image $\mc P(A',\vec b' k+ \vec c_{(\ell)})$ is also bounded.
 By \eqref{peterd} we have
\begin{equation} \label{n.break.up}
N(A,\vec bk+\vec c)=N(A,\vec bk+\vec {c'}) + \sum_{\ell=1}^{c_1} N(A',\vec b'k+ \vec c_{(\ell)}),
\end{equation}
where  $\vec {c'}=(0,c_2,c_3,\ldots)^t$. 

By our induction on nonzero components of $\vec c$, we know that for $k \gg 0$, $N(A,\vec bk+\vec {c'})=  g(k)$, where $g$ is a polynomial with $\deg g \leq n$.
Moreover, if  $\dim \mc P(A,\vec b)=n$, then $\deg g=n$, and its leading coefficient is the  volume of $\mc P(A,\vec b)$.

For a given $1 \leq \ell \leq c_1$, either $\mc P(A',\vec b'k+ \vec c_{(\ell)}) = \emptyset$ for $k \gg 0$, or 
\beq
\mc P(A',\vec b'k+ \vec c_{(\ell)})=\mc P(A'_*,\vec b'_*k+ \vec c_{(\ell),*})
\eeq
by Lemma  \ref{zero.row.lemma}. In the first case, put $f_{(\ell)}=0$. In the second case,  $A'_*$ has no zero rows. Therefore by our induction on $n$,  there are polynomials $f_{(\ell)}$, with $\deg  f_{(\ell)} \leq n-1$, so that for $k \gg 0$ we have $N(A',\vec b'k+\vec c_{(\ell)})=f_{(\ell)}$.   
Now
\beq
f=g+\sum_{\ell=1}^{c_1} f_{(\ell)}
\eeq
satisfies the conclusion of the theorem.

For the case where $c_1<0$, consider \eqref{eq: 1} and \eqref{eq: 2}, but with  $\ell = c_1+1, \ldots, 0$. The elimination procedure runs as before, but replacing \eqref{n.break.up} with
\beq
N(A,\vec bk+\vec c) +  \sum_{\ell=c_1+1}^{0} N(A',\vec b' k+ \vec c_{(\ell)})
           =N(A,\vec bk+\vec {c'}),
           \eeq
and one takes
\beq
f=g-\sum_{\ell=c_1+1}^{0} f_{(\ell)}.
\eeq
(This time $\mc P^\ell(k)$ is bounded because it is contained in $\mc P(A,\vec bk+\vec{c'})$; thus its projection $\mc P(A',\vec b' k+\vec c_{(\ell)})$ is bounded.)

This completes the induction, and the theorem is proved. \end{proof}

\begin{ex}
Let $A=\begin{pmatrix} 0 \\ 1 \\ \end{pmatrix}$, $\vec b=\begin{pmatrix} 0 \\ 1 \\ \end{pmatrix}$, and  $\vec c=\begin{pmatrix} -1 \\ 0 \\ \end{pmatrix}$. Then $\mc P(A,\vec b)$ is the interval $[0,1]$, but $\mc P(A,\vec bk+\vec c)=\emptyset$. So we cannot remove the hypothesis in Theorem \ref{totunimod polytope polynomial} that none of the rows of $A$ are $0$.
\end{ex}

\begin{ex} 
Let $A= \begin{pmatrix} 1 &0 \\ -1 & 0\\  0 & 1 \\ 0 & -1 \\   \end{pmatrix}$, $\vec b= \begin{pmatrix} 1 \\ -1 \\ 1 \\ 0 \end{pmatrix}$, and $\vec c= \begin{pmatrix} 0 \\ -1 \\  0 \\ 0 \end{pmatrix}$. Then $\mc P(A,\vec b)=\{ 1\} \times [0,1]$, but $\mc P(A,\vec b k+ \vec c)=\emptyset$ for all $k$. So we cannot remove the hypothesis in Theorem \ref{totunimod polytope polynomial} that  $\dim(\mc P(A,\vec b)) = n$.
\end{ex}

Write $M_{\vec{r}, \vec{c}}$ for the number of nonnegative integer points in the transportation polytope $\mc M(\vec r, \vec c)$ and $V_{\vec r, \vec c}$ for its relative volume.

\begin{lemma}
\label{lemma: Mrkck polynomial} For $1 \leq i \leq s$ let $r_i,a_i \in \Z$, and for $1 \leq j \leq t$ let $b_j,c_j \in \Z$, with
 $\sum r_i = \sum c_j$ and $\sum a_i =\sum b_j$. Put $\vec{r}_k = (r_1k + a_1,  \ldots, r_sk + a_s)$ and $\vec{c}_k = (c_1k + b_1, \ldots, c_tk + b_t)$. Then for $k \gg 0$, the function $k \mapsto M_{\vec{r}_k, \vec{c}_k}$ is a polynomial in $k$ of degree equal to $(s-1)(t-1)$ and leading coefficient $V_{\vec{r}, \vec{c}}$.

\end{lemma}
\begin{proof}
The nonnegative integer points in $\mc M(\vec{r}_k$, $\vec{c}_k)$ are in bijection with the nonnegative integer points in its projection  $\mc M'(\vec{r}_k, \vec{c}_k)$. By Proposition \ref{transport.A.b.form}, we may write
\beq
\mc M'(\vec{r}_k, \vec{c}_k)=\mc P(A,\vec b k+\vec c),
\eeq
for an $(s+t-1) \times (s-1)(t-1)$ matrix $A$, and $\vec b,\vec c \in \R^{st}$. Moreover, $A$ is totally unimodular with no zero rows, and $\mc P(A,\vec b)$ is bounded, of dimension $(s-1)(t-1)$. 
 Hence the result follows by Theorem \ref{totunimod polytope polynomial}.
\end{proof}

\section{Proofs of the Main Theorems}\label{sec: proofs of thm 1 and 3}

All of the above was aimed towards proving Theorems  \ref{main.thm.intro} and \ref{not.prime.intro}, which we complete in this section. 

Recall that for integral vectors $\vec r$ and $\vec c$, write $M_{\vec r, \vec c}$ for the number of nonnegative integral matrices with row margins $\vec r$ and column margins $\vec c$. Let us write $V_{s,t}$ for the relative volume of the transportation polytope for row margins $(\underbrace{s, \ldots, s}_{t \text{ times}})$ and column margins $(\underbrace{t, \ldots, t}_{s \text{ times}})$. 
For instance, by Example \ref{Ex.V23}, $V_{2,3}=3$.

We recall Theorem \ref{main.thm.intro} from the Introduction.
\bigskip

\noindent {\bf Theorem 1.}  Let $s, t$ be relatively prime.  There is a quasipolynomial $Q_{\sigma,\tau}(k)$ of degree $(s-1)(t-1)$ and period $st$, so that for integers
 $k \gg 0$, we have $\Nstk=Q_{\sigma,\tau}(k)$.  The leading coefficient of $Q_{\sigma,\tau}(k)$ is $V_{s,t}$.
 
\begin{proof} 
 Put
 \beq
 (r_1, \ldots, r_s)=kt(1,1, \ldots,1)
 \eeq
 and
 \beq
 (c_1, \ldots, c_t)=ks(1,1,\ldots,1).
 \eeq
 By Theorem  \ref{rel.prime.cor.today}, there are integers $a_1,\ldots, a_s$ and $b_1, \ldots, b_t$ so that if
  $\vec{r}_k = (r_1k + a_1,  \ldots, r_sk + a_s)$ and $\vec{c}_k = (c_1k + b_1, \ldots, c_tk + b_t)$, then
 \beq
N_{\sigma,\tau}(i+stk)= M_{\vec r_k,\vec c_k}.
\eeq
for $i \gg 0$. The conclusion then follows from Lemma   \ref{lemma: Mrkck polynomial}.
\end{proof}

\begin{ex}
Let $s =2, t=3$, $\sigma = \tau = \emptyset$.  Then
$$N_{\emptyset, \emptyset}(6k) = M_{\vec r_k, \vec c_k}$$ where $\vec r_k = (3k,3k)$ and $\vec c_k = (2k,2k,2k)$. The projection $\mc M'(\vec r_k,\vec c_k)$ is illustrated in Figure 2.

\begin{figure}[!h]\label{Figure 2}
\begin{tikzpicture}[scale=2]
\draw (0,1) -- (1,0) -- (2,0) -- (2,1) -- (1,2) -- (0,2) -- (0,1);
\node[label=left:{$(0,k)$}] at (0,1) {};
\node[label=left:{$(k,0)$}] at (1,0) {};
\node[label=right:{$(2k,0)$}] at (2,0) {};
\node[label=right:{$(2k,k)$}] at (2,1) {};
\node[label=right:{$(k,2k)$}] at (1,2) {};
\node[label=left:{$(0,2k)$}] at (0,2) {};
\end{tikzpicture}
\caption{The transportation polytope for $i=0$}
\end{figure}

In fact, $N_{\emptyset, \emptyset}(6k) = 3k^2+3k+1$. For $i =1$, $$N_{\emptyset, \emptyset}(1+6k) = M_{\vec r_k, \vec c_k}$$ where $\vec r_k = (3k+1,3k)$ and $\vec c_k = (2k+1,2k,2k)$. The projection $\mc M'(\vec r_k, \vec c_k)$ is given in Figure 3.

\begin{figure}\label{Figure 3}
\begin{tikzpicture}[scale=2]
\draw (3,0) -- (5,0) -- (5,2) -- (3,4) -- (0,4) -- (0,3) -- (3,0);
\node[label=left:{$(k+1,0)$}] at (3,0) {};
\node[label=right:{$(2k+1,0)$}] at (5,0) {};
\node[label=right:{$(2k+1,k)$}] at (5,2) {};
\node[label=right:{$(k+1,2k)$}] at (3,4) {};
\node[label=left:{$(0,2k)$}] at (0,4) {};
\node[label=left:{$(0,k+1)$}] at (0,3) {};
\end{tikzpicture}
\caption{The transportation polytope for $i=1$}
\end{figure}

From this, one computes 
\beq
N_{\emptyset, \emptyset}(n)=\left\{
\begin{array}{@{}ll@{}}
3k^2 + 3k + 1, & \text{if}\ n=6k \\
3k^2 + 4k + 1, & \text{if}\ n=6k + 1 \\
3k^2 + 5k + 2, & \text{if}\ n=6k +2 \\
3k^2 + 6k + 3, & \text{if}\ n=6k +3 \\
3k^2 + 7k + 4, & \text{if}\ n=6k +4 \\
3k^2 + 8k + 5, & \text{if}\ n=6k +5.
\end{array}\right.
\eeq 
\end{ex}

We recall Theorem \ref{not.prime.intro} from the Introduction.
\bigskip

\noindent {\bf Theorem 3.}
If $\core_d(\sigma) = \core_d(\tau)$, then there is a quasipolynomial $Q_{\sigma, \tau}(k)$ of degree $\frac{1}{d} (s-d)(t-d)$ and period $m$, so that for integers $k \gg 0$, we have $\Nstk=Q_{\sigma,\tau}(k)$. The leading coefficient of $Q_{\sigma, \tau}(k)$ is $\left(V_{\frac{s}{d},\frac{t}{d}}\right)^d$.

\begin{proof}
Let $i \geq \ell_0$. We must show that for $k \gg 0$, the map $k \mapsto N_{\sigma,\tau}(i+mk)$ is polynomial of the given degree and leading coefficient.
By Theorem \ref{prop: Nst = prodNsigmajtauj}, 
\beq
N_{\sigma,\tau}(i+mk)= \prod_{j=0}^{d-1} N_{\sigma_{j}^i,\tau_{j}^i} \left(\ell_{j}^i + \frac{m}{d}k \right),
\eeq
where each $\sigma_j^i$ is an $\frac{s}{d}$-core, each $\tau_j^i$ is a $\frac{t}{d}$-core, and $\ell_j^i$ are certain nonnegative integers. Recall that $\frac{st}{d^2}=\frac{m}{d}$. By Theorem \ref{main.thm.intro}, for each $i,j$, and with $k \gg 0$, 
the map $k \mapsto N_{\sigma_{j}^i,\tau_{j}^i} \left(\ell_{j}^i + \frac{m}{d}k \right)$ is a polynomial of degree 
\beq
 \left( \frac{s}{d}-1 \right) \left( \frac{t}{d}-1 \right),
 \eeq
 and  leading coefficient $V_{\frac{s}{d},\frac{t}{d}}$. 
 The theorem follows. 
\end{proof}

\begin{ex}\label{ex2:noncoprime case}
From Example \ref{example1: noncoprime} we have for $\sigma = (3,1,1), \tau = (3,2)$, $$N_{\sigma, \tau}(12 + 12k) = N_{(1),(1,1)}(7+6k) \cdot N_{(1),\emptyset}(5+6k).$$
Following the proof of Theorem~\ref{totunimod polytope polynomial} gives
\beq
N_{(1),(1,1)}(7+6k)= 3k^2 +10k+7 
\eeq
and
\beq
N_{(1),\emptyset}(5+6k) = 3k^2 +8k+5.
\eeq
 Thus  
 \beq
 N_{\sigma, \tau}(12 + 12k) = (3k^2 +10k+7 )(3k^2 +8k+5).
 \eeq
 
\end{ex}
 
 Finally, we consider the number of $\la$ with length \emph{exactly} $k$, and having $s$-core $\sigma$ and $t$-core $\tau$. This is  given by
 \beq
 N'_{\sigma, \tau}(k) = \Nstk - N_{\sigma, \tau}(k-1).
\eeq

From  Theorem \ref{not.prime.intro}  we deduce:

\begin{cor}
There is a quasipolynomial $Q'_{\sigma, \tau}(k)$ of degree less than $\frac{1}{d} (s-d)(t-d)$ and period $m$, so that for integers $k \gg 0$, we have $N'_{\sigma,\tau}(k)=Q'_{\sigma,\tau}(k)$. 
\end{cor}

 \bibliographystyle{alpha}
\bibliography{References}
\end{document}